\documentclass[12pt]{amsart} 

\usepackage{amssymb}
\usepackage{amscd}
\usepackage{xcolor} 
\usepackage{graphicx} 
\usepackage{enumerate}

\usepackage{amsmath}
\usepackage{amssymb}
\usepackage{amsrefs}
\usepackage{ulem}

\usepackage[displaymath,mathlines, pagewise]{lineno}

\usepackage
{hyperref}
\usepackage{enumitem}
\hypersetup{
  colorlinks  = true,    
  urlcolor    = red,    
  linkcolor   = blue,    
  citecolor   = blue,      
  pdfauthor   = {Hector Alonzo Barriga Acosta},%
  pdfsubject  = {Protocolo},%
  pdftitle    = {Countable Nabla Products},  %
}

\theoremstyle{plain}
\newtheorem{theorem}{Theorem}[section]
\newtheorem{lemma}[theorem]{Lemma}

\newtheorem{observation}[theorem]{Observation}
\newtheorem{maintheorem}[theorem]{Main Theorem}

\theoremstyle{definition}

\newtheorem{notat}{Notation}[section]
\newtheorem{question}{Question}[section]

\newtheorem*{case1}{Case One}
\newtheorem*{case2}{Case Two}

\newtheorem{definition}[theorem]{Definition}
\newtheorem{remark}{Remark}[section]

\theoremstyle{Note}

\newtheorem{fact}{Fact}

\numberwithin{equation}{section}

\DeclareMathOperator{\val}{\mathop{val}}

\title{On the weak pseudoradiality of CSC spaces
}

\author[H. Barriga-Acosta]{H. A. Barriga-Acosta}
\address{Department of Mathematics,
University of North Carolina at Charlotte, 
Charlotte, NC 28223}
\email{hbarriga@uncc.edu}

\author[A. Dow]{A. Dow}
\address{Department of Mathematics,
University of North Carolina at Charlotte, 
Charlotte, NC 28223}
\email{adow@uncc.edu}
\date{\today}
\keywords{Sequentially compact,
 pseudoradial, weakly pseudoradial, Davies-trees, pseudoradial number, finally property K, elementary submodel}

\subjclass{03E04, 03E10, 03E17, 03E35, 03E65, 06E15, 54A20, 54A35, 54B10 }

\begin{document}
\begin{abstract}
In this paper we prove that in forcing extensions by a poset with finally property K over a model of GCH+$\square$, every compact sequentially compact space is weakly pseudoradial. This improves Theorem 4 in \cite{dow1996more}. We also prove the following assuming $\mathfrak{s}\leq \aleph_2$: (i) if $X$ is compact weakly pseudoradial, then $X$ is pseudoradial if and only if $X$ cannot be mapped onto $[0,1]^\mathfrak{s}$; (ii) if $X$ and $Y$ are compact pseudoradial spaces such that $X\times Y$ is weakly pseudoradial, then $X\times Y$ is pseudoradial. This results add to the wide variety of partial answers to the question by Gerlits and Nagy of whether the product of two compact pseudoradial spaces is pseudoradial.
\end{abstract}
\maketitle

\bibliographystyle{plain}

\section{Introduction}
A space is {\it sequentially compact} if every countable sequence
has a converging subsequence. Following \cite{dow1996more},
say that a space is CSC if it is compact and sequentially compact. 
A subset $A$ of a space $X$ is {\it radially closed} if there is no sequence $\{x_\alpha : \alpha < \kappa\} \subseteq A$ that converges to a point in $X\setminus A$ (here ``converges'' means that each neighborhood of the limit point leaves out $< \kappa$-many members of the sequence, hence we can assume $\kappa$ is regular). 
The {\it radial closure} of $A$ is the minimal radially closed set $A^{(r)}$ that contains $A$. A space is {\it pseudoradial} if the radial closure of every subset is closed. 

\medskip 
The {\it splitting number} $\mathfrak{s}$, which is equal to 
$\min \{ \kappa : 2^\kappa$ is not sequentially compact$\}$, plays an important role regarding pseudoradial spaces. 
It is well known that $2^{\omega_1}$ is pseudoradial if and only if $\mathfrak{s} > \omega_1$. 
Analogously we can define the {\it pseudoradial number}, $\mathfrak{pse} = \min \{ \kappa : 2^\kappa$ is not pseudoradial$\}$. Then $\mathfrak{s}> \omega_1$ implies $\mathfrak{pse} > \omega_1$ (hence $\mathfrak{pse} = \omega_1$ implies $\mathfrak{s} = \omega_1$). Moreover, since every compact pseudoradial space is sequentially compact we have $\mathfrak{pse \leq s}$.
It is unclear to the authors whether $\mathfrak{pse}$ is regular or can have countable cofinality. 

\medskip
In \cite{juhasz1993sequential} Juhász and Szentmiklóssy proved that (i) assuming $\mathfrak c\leq\aleph_2$, every CSC space is pseudoradial (this improves the result in \cite{shapirovskii1990ordering} by \v Sapirovski\v\i \ who assumed CH).
It was also shown there that (ii) a compact non-pseudoradial space contains a subset of size less than $\mathfrak{c}$ whose closure is not pseudoradial.
Further, they proved that (iii) there is a model of $\mathfrak{c}= \aleph_3$ in which there is a CSC non-pseudoradial space, and asked whether $\mathfrak{c}= \aleph_3$ implies the existence of such spaces.
In \cite{dow1996more} Dow, Juhász, Soukup and Szentmiklóssy improved (ii) by replacing $\mathfrak{c}$ for $\mathfrak{s}$, and they used this fact to show that (iv) in the extension by adding any number of Cohen reals to a model of CH, every CSC space is pseudoradial. 
This solves in the negative to the question from (iii).
Later, in \cite{bella2001pseudoradial} Bella, Dow and Tironi  focused mainly on whether a compact non-pseudoradial space necessarily contains a closed separable non-pseudoradial subspace. They showed that this is consistently true: if $2^{\omega_2}$ is not pseudoradial, then a compact space is pseudoradial if eve\-ry closed separable subspace is pseudoradial. The following question remains open.

\begin{question}[\v Sapirovski\v \i]\label{Q:Sapirovskii}
Is it true in ZFC that $2^{\omega_2}$ is not pseudoradial?
\end{question}

A weaker property than ``all closed separable subspaces are pseudoradial'' is the following. 
A space is {\it weakly pseudoradial} if the radial closure of every countable subset is closed. The work in this paper is motivated by the facts stated above and the target is to study weak pseudoradiality. 
It turns out that under the presence of $\mathfrak{pse}=\aleph_2$, weak pseudoradiality provides a nice equivalence of pseudoradiality. In Section~\ref{SectionProducts} we  prove the following

\begin{theorem}\label{t:pseudoradialiff}
Suppose $\mathfrak{pse}\leq\aleph_2$. Let $X$ be a compact weakly pseudoradial space. Then, $X$ is pseudoradial if and only if $X$ cannot be mapped onto $[0,1]^\mathfrak{pse}$.
\end{theorem}

\medskip
A poset $\mathbb{Q}$ is {\it linked} provided its members are pairwise compatible. 
A  subposet $\mathbb{Q}\subseteq \mathbb{P}$ is {\it complete} in $\mathbb{P}$ is every maximal antichain of $\mathbb{Q}$ is maximal in $\mathbb{P}$.
We say that a poset $\mathbb{P}$ has {\it finally property K} if for every complete subposet $\mathbb{Q}\subseteq \mathbb{P}$, the factor poset $\mathbb{P}/\mathbb{Q}$ (see \cite{kunen2014set}) is forced by $\mathbb{Q}$ to have property K (every uncountable subset has an uncountable linked subset) as in  \cite{dow2014reflecting}. As pointed out in Section~\ref{SectionPreliminaries}, posets with finally property K are $ccc$, the Cohen forcing has finally property K, and finite support iterations (products) of posets with finally property K have finally property K. 
Now we state the central result of this document whose proof is in Section~\ref{MainResult}.

\begin{maintheorem}\label{MainTheorem}
Assume $V\models$\textup{GCH+}$\square$. Suppose that $\mathbb{P}$ is a poset with finally property K and $G\subseteq \mathbb{P}$ is a $\mathbb{P}$-generic filter. Then, in $V[G]$, every CSC space is weakly pseudoradial.
\end{maintheorem}

Let us observe that if we further assume $V[G] \models \mathfrak{s}\leq\aleph_2$, then in the extension every CSC is pseudoradial: if $X$ is CSC then it is weakly pseudoradial by Main Theorem \ref{MainTheorem}, and in particular it is sequentially compact. Now observe that $\mathfrak{s}\leq \aleph_2$ implies $\mathfrak{pse} = \mathfrak{s}$. Thus, $X$ cannot be mapped onto the non-sequentially compact space $[0,1]^\mathfrak{pse}$. Theorem \ref{t:pseudoradialiff} applies, so $X$ is pseudoradial.
Recall that in forcing extensions by adding Cohen reals we have $\mathfrak{s}=\aleph_1$.
Subsequently Main Theorem \ref{MainTheorem} generalizes result (iv) stated above (Theorem 4  in \cite{dow1996more}).
\bigskip

In a different direction, one of the main problems in the theory of pseudoradial spaces is due to Gerlits and Nagy (\cite{gerlits}) who asked whether the product of two compact Hausdorff pseudoradial spaces is pseudoradial. Many partial results have been given, though the question remains open in ZFC. In \cite{tironi1989products} Frolik and Tironi proved that the product of two compact Hausorff pseudoradial spaces is pseudoradial if one of them is radial. This was improved by Bella and Gerlits in \cite{bella2001pseudoradial} by only requiring one of the factors to be semi-radial. In \cite{bella1993countable} Bella proved that the the product of countably many compact Hausdorff $R$-monolitic spaces is $R$-monolitic. As a consequence of Juhász and Szentmiklóssy result, if $\mathfrak{c}\leq \aleph_2$ then the product of countably many pseudoradial spaces is pseudoradial. In \cite{obersnel1995products} Obersnel and Tironi showed assuming $\mathfrak{h} \leq \aleph_3$ that for any $\kappa < \mathfrak{h}$, if $\{ X_\alpha : \alpha < \kappa \}$ is a family of compact Hausdorff pseudoradial spaces with $|X_\alpha| < 2^{\omega_2}$, then $\prod_{\alpha < \kappa} X_\alpha$ is pseudoradial.

\medskip
We use Theorem \ref{t:pseudoradialiff} and Lemma \ref{LemmaProductCannotBeMapped} to prove the next result and we leave a natural question from it.

\begin{theorem}\label{t:product}
Suppose $\mathfrak{pse} \leq \aleph_2$. Let $X$ and $Y$ be compact pseudoradial spaces such that $X\times Y$ is weakly pseudoradial. Then $X\times Y$ is pseudoradial.
\end{theorem}

\begin{question}
Is it true in ZFC that the product of two compact pseudoradial spaces is weakly pseudoradial?
\end{question}

\section{Preliminaries}\label{SectionPreliminaries}

\subsection{Topology}\label{SubTopology}

We follow notation from \cite{juhasz1993sequential}. Let $X$ be a space and $A\subseteq X$ be a non-closed subset. Define
$$\lambda (A, X) = \min \{ \lambda : \exists K\subseteq \overline{A} \text{ a non-empty closed } G_\lambda\text{-set } (K \cap A = \emptyset) \}.$$
Note that if $K$ is a $G_\lambda$-set witness of $\lambda = \lambda (A,X)$, then by the minimality of $\lambda$ there is a sequence $\{ x_\alpha : \alpha < \lambda\} \subseteq A$ converging to $K$, that is, every open set containing $K$ also contains a final segment of $\{ x_\alpha : \alpha < \lambda\}$.
Moreover, if $X$ is sequentially compact and $A$ is radially closed then $\lambda (A,X)$ has uncountable cofinality.  
\medskip


\begin{observation}\label{Observation:implications}
Let $X$ be compact. Then,
\begin{enumerate}
    \item ``$X$ is pseudoradial'' implies
    \item ``all closed separable subspaces of $X$ are pseudoradial''  implies
    \item  ``$X$ is weakly pseudoradial'' implies
    \item  ``$X$ is sequentially compact''.
\end{enumerate}
\end{observation}

Under $\mathfrak{c} \leq \aleph_2$, every CSC space is pseudoradial, hence the preceding properties are equivalent. We find it interesting to expand the discussion on Observation~\ref{Observation:implications}. 

\smallskip 
A key lemma in \cite{juhasz1993sequential} is: if $X$ is CSC then for every non-closed set $A\subseteq X$, $\omega < \lambda(A, X) < \mathfrak{c}^-$, where $\mathfrak{c}^-$ is equal to $\mathfrak{c}$ in case it is limit; otherwise, it is the predecessor of $\mathfrak{c}$. 
Note that $\mathfrak{p=c}$ suffices to prove (4) implies (3): if $A$ is countable non-closed, then there is $K \subseteq \overline{A} \setminus A$ a closed $G_\lambda$-set, where $\lambda = \lambda (A,X)$. 
This produces a centered family on the countable set $A$ of size $\lambda < \mathfrak{p}$, hence it has a pseudointersection. Because $X$ is sequentially compact, the pseudointersection has a subsequence converging to some point in $K$. This contradicts $A$ is non-closed. 

It was also proven in \cite{juhasz1993sequential} that if the c.u.b. filter on $\omega_1$ has character $\kappa$, then $2^\kappa$ is not pseudoradial. Note that $2^\omega = \omega_3$ is consistent with MA plus `the cub filter on $\omega_1$ has character $\omega_2$'. In this model we have on one hand, $\mathfrak{p=s=c}=\omega_3$ which implies $2^{\omega_2}$ is separable, sequentially compact and weakly pseudoradial. On the other hand, $2^{\omega_2}$ is not pseudoradial. That is, in this model (3) does not imply (2). We leave the questions regarding the rest of the implications.

\begin{question}
Is it consistent with ZFC that there exists a CSC non-weakly pseudoradial space?
\end{question}

\begin{question}
Is it consistent with ZFC that there exists a compact non-pseudoradial space in which all closed separable subspaces are pseudoradial? 
\end{question}

For the last question necessarily $2^{\omega_2}$ must be pseudoradial due to Bella-Dow-Tironi \cite{bella2001pseudoradial}. This would answer in the positive to Question \ref{Q:Sapirovskii}.

\bigskip
If $x\in X$, a {\it $\pi$-base} of $x$ is a family $\mathcal{U}$ of non-empty open sets of $X$ such that every neighborhood of $x$ contains a member of $\mathcal{U}$. The {\it $\pi$-character of $x$ in $X$} is $\pi \chi (x, X) = \min \{|\mathcal{U}| : \mathcal{U} \ \mbox{is a} \ \pi \mbox{-base of} \ x \}$, and the {\it $\pi$-character of $X$} is $\pi \chi (X) = \sup \{ \pi \chi (x, X) : x\in X \}$. In Section \ref{SectionProducts} we will use these notions as well as the following result in \cite{juhasz1980cardinal}.

\begin{theorem}[\v Sapirovski\v \i]\label{t:Sapirovskii}
The following are equivalent for a compact space $X$:
\begin{enumerate}
    \item[i)] $X$ can be continuously mapped onto $I^\kappa$;
    \item[ii)] there is a closed set $F\subseteq X$ which can be continuously mapped onto $2^\kappa$;
    \item[iii)] there is a closed set $F\subseteq X$ with $\pi \chi (x, F) \geq \kappa$ for each $x\in F$.
\end{enumerate}
\end{theorem}

\medskip

\subsection{Elementary Submodels}\label{SubS-ESM}
For the proof of the Main Theorem~\ref{MainTheorem} we will make heavy use of elementary submodels $M$ of $H(\theta)$, where $\theta$ is a large enough cardinal. We will also use the following properties about finally property K posets and extensions  by generic filters over structures.
\smallskip

It is well known (see \cite{dow2014reflecting}, \cite{kunen2014set}) that a finally property K poset $\mathbb{P}$ is $ccc$ ($\mathbb{Q} = \{ 1 \}$ is a complete subposet of $\mathbb{P}$ and $\mathbb{P}/\mathbb{Q} \simeq \mathbb{P}$).
Moreover, if $M\prec H(\theta)$, $\mathbb{P}\in M$ and $M^\omega\subseteq M$, then $\mathbb{P}_M = \mathbb{P}\cap M$ is a complete subposet of $\mathbb{P}$. %
This  implies that maximal antichains of $\mathbb{P}_M$ are  maximal antichains of $\mathbb{P}$, and it also implies that, if $G$ is a $\mathbb{P}$-generic filter, then $V[G]$ is obtained by forcing with $\mathbb{P}/\mathbb{P}_M = \mathbb{P}/(G\cap M)$ over the model
$V[G_M]$, where $G_M = G\cap \mathbb{P}_M$. 
Recall the facts (see \cite{dow2002recent}) that $M[G_M] \cap \mathcal{P}(\omega) = M[G] \cap \mathcal{P}(\omega)$, $M[G_M]$ is an elementary submodel of $H(\theta)[G_M] $ (this is simply $H(\theta)$ in the sense of $V[G_M]$), and that $M[G]$ (hence $M[G_M]$) is also closed under $\omega$-sequences in the universe $V[G_M]$.

\subsection{Trees}\label{SubTrees}
Here we introduce an important tool (a tree) that will be used in Lemmas \ref{l:*} and \ref{l:**}. 
In \cite{soukup2018infinite} Dániel Soukup and Lajos Soukup defined and contructed from the Jensen's principle $\square$ the {\it high} and {\it sage Davies-trees}. We opt to only state what we need from these trees.
\smallskip

Suppose GCH and $\square$ hold. Let $\kappa$ be a cardinal such that $\kappa^\omega = \kappa$ and let $x$ be any set. Then it is possible to recursively construct a tree $T_\kappa$ together with models $M_t$, for $ t\in T$, with the following requirements. 
The elements $t$ of $T_\kappa$ are finite functions with domain an integer into successor ordinals. 
The model $M_\emptyset$ will be the increasing union of its immediate successors and will have size $\kappa$.
Let $\kappa_t$ denote the cardinality of $M_t$.
Here we list the required properties about the tree $T_\kappa$:
\begin{enumerate}
    \item if $\kappa=\aleph_1$ then every $M_\alpha$ is countable;
    \item a node $t$ of $T_\kappa$ is maximal if and only if $M_t$ is countable;
    \item for every $t\in T_\kappa$, $x\in M_t$;
    \item given $t\in T_\kappa$, the sequence $\{ M_{t^\frown (\alpha+1)} : \alpha <\mbox{cf}(|M_t|)\}$ is a $\subseteq$-chain that unions up to $M_t$, and $\kappa_{t^\frown (\alpha+1)} < \kappa_t$;
    \item if $\kappa_t = \lambda^+$ with $\mbox{cf}(\lambda)= \omega$, then for every $\alpha < \mbox{cf}(\kappa_t)$, $M_{t^\frown (\alpha+1 ^\frown n)}$ is closed under $\omega$-sequences and $\kappa_{t^\frown (\alpha+1 ^\frown n)}$ is regular,  for each $n\in\omega$;
    \item  if $\kappa_t$ is any other cardinal, then $M_{t^\frown (\alpha+1)}$ is closed under $\omega$-sequences for all $\alpha<\mbox{cf}(\kappa_t)$, and $\kappa_{t^\frown (\alpha+1)} = \kappa_{t^\frown (\beta+1)}$.
\end{enumerate}

In \cite{soukup2018infinite}, clause (II) in the definition of high Davies-tree implies that $M_\emptyset$ has size $\kappa$ and is closed under $\omega$-sequences. In the proof of Theorem 14.1 \cite{soukup2018infinite}, their models $K_{\alpha\!+1}$ for Case I are the models $M_{t^\frown (\alpha+1)}$ for our item (5), and their models $K_{\alpha\!+1, j}$ for Case II are the models $M_{t^\frown (\alpha+1 ^\frown j)}$ for our item (4). 
Observe that we are only considering models that have successor index; if the index value for a model is limit we could not guarantee that the model is closed under $\omega$-sequences.
\medskip

Clearly $T_\kappa$ has no infinite branches and this is equivalent (see \cite{kunen2014set}) to saying that $T_\kappa$, with the reverse ordering, is well-founded. 
There is a rank function, $\mathop{rk}_{T_\kappa}$, on $T_\kappa$ where
$\mathop{rk}_{T_\kappa}(t)=0$ if $t$ is ma\-xi\-mal. 
For non-maximal $t$, the definition of $\mathop{rk}_{T_\kappa}(t)$ is minimal so that $\mathop{rk}_{T_\kappa}(t^\frown \alpha) < \mathop{rk}_{T_\kappa}(t)$ for all $t^\frown \alpha\in T_\kappa$.

\section{The main result}\label{MainResult}
We want to prove that if we force with a finally property K poset $\mathbb{P}$ over a model of GCH + $\square$, then in the extension every CSC space is weakly pseudoradial. 
We will present the proofs and results for  0-dimensional spaces  and leave the routine  changes needed to handle the  general case to the interested reader. 
So, we focus on separable 0-dimensional CSC spaces and for practical purposes we identify any countable dense set with $\omega$.
If $X$ is a 0-dimensional CSC space with dense set $\omega$ then there is a Boolean algebra $B_X$ on $\omega$ whose Stone space $S(B_X)$ is $X$ ($B_X$ is the Boolean algebra of the clopen sets of $X$ intersected with $\omega$).
\medskip

Throughout this section suppose $V$ is a model of GCH + $\square$, $\mathbb{P}$ has finally property K, $G$ is a $\mathbb{P}$-generic filter and, in $V[G]$, let $X$ be a separable 0-dimensional CSC space with dense set $\omega$. 
Let $\dot B_X$ be a family of nice $\mathbb{P}$-names of subsets of $\omega$ that is forced, by 1, to be the Boolean algebra on $\omega$ whose Stone space is $X$; 1 forces that $S(\dot B_X) = \dot X$ is CSC.
We may assume that the fixed ultrafilters of $\dot B_X$ are the elements of $\omega$ and that for all $n\neq m\in\omega$, there is a $\dot b\in \dot B_X$ satisfying that $1\Vdash |\dot b\cap \{n,m\}|=1$ (i.e. $\omega$ is dense but not necessarily discrete or open). 
\medskip

As suggested, we aim to prove that in the forcing extension the radial closure of $\omega$ (the countable dense set in $X$) is closed.
That is, if $\dot u$ is a $\mathbb{P}$-name for an ultrafilter on $\omega$ ($1\Vdash \dot u\in \omega^*$), prove that the $\dot u$-limit of $\omega$ is in $\omega^{(r)}$. 
To this end here is the key idea: we will get the desired $\dot u$-limit as being the limit of a well-ordered sequence of points in the radial closure of $\omega$, and these points are produced by using larger and larger elementary submodels. More concretely, we will use induction over $rk_{T_\kappa}$, for large enough $\kappa$, to get points in $\omega^{(r)}$ as in Definition \ref{d:*} until we obtain a converging sequence to the $\dot u$-limit.

\begin{notat}\label{n:A(u,W)}
Suppose $\dot u$ is a $\mathbb{P}$-name for an ultrafilter on $\omega$ and $\dot B$ is a list of $\mathbb{P}$-names such that 1 forces $\dot B$ is a Boolean algebra on $\omega$. 
For any countable family $\mathcal W$ of $\dot u$,
let $\mathcal A(\dot u,\mathcal W)$ denote the family of all nice $\mathbb{P}$-names $\dot a$ (of subsets of $\omega$) where 1 forces that $\dot a\subseteq^* \dot W$ for all $\dot W\in\mathcal W$, and $\dot a$ is a converging sequence in $S(\dot B)$. 
Of course $\mathcal A(\dot u,\emptyset)$ contains $\mathcal A(\dot u,\mathcal W)$ for all countable  $\mathcal W\subseteq \dot u$. For any $\dot a\in \mathcal A(\dot u,\emptyset)$, let $x_{\dot a}$ denote the limit point of $\dot a$ in $S(\dot B)$.
\end{notat}

We may think of $\mathcal A(\dot u,\mathcal W)$ as the collection of all sequences that converges to the $G_\delta$-set, $\bigcap \mathcal{W}$, which contains the $\dot u$-limit. By sequential compactness, these sets are non-empty.


\begin{lemma}\label{KeyLemma}
Suppose $M \prec H(\theta)$ has size $\aleph_1$, is closed under $\omega$-sequences and is the increasing union of a sequence of countable elementary submodels $\langle M_{\alpha} : \alpha\in\omega_1\rangle$. For each $\alpha\in\omega_1$, choose $\dot a_\alpha$ an element of $M\cap \mathcal A(\dot u,\left( M_\alpha\cap \dot u\right))$. Then there is a point $\dot x(\dot u, M)$ in the radial closure of $\omega$ satisfying that the sequence $\langle x_{\dot a_\alpha} : \alpha\in  \omega_1\rangle$ converges to $\dot x(\dot u,M)$. Moreover, $\dot x (\dot u, M)$ does not depend on the choice of the $\dot a_\alpha$'s.
\end{lemma}

\begin{proof}
We have the sequence $\{ \val_{G}(\dot a_\alpha): \alpha\in\omega_1\}$ in the model $V[G_M]$. (This sequence is not necessarily in the elementary submodel $M[G]$ as it is not required to be closed under $\omega_1$-sequences, see Subsection \ref{SubS-ESM}.)

\begin{fact}
$\dot u_{M} =  \{ \val_{G}(\dot U) : \dot U\in
\dot u\cap M\}$ is a $\mathbb{P}_M$-name of an ultrafilter on $\omega$ (i.e. $\val_{G_M}(\dot u_M) = \val_{G}(\dot u)\cap V[G_M]$ is an ultrafilter on $\omega$).
\end{fact}
First let us observe that $\mathcal{P}(\omega) \cap V[G_M] \subseteq M[G_M]$. In fact, if $\dot C$ is $\mathbb{P}_M$-nice name for a subset of $\omega$ in $V[G_M]$ then $\dot C$ is a countable subset of $\omega\times \mathbb{P}_M$ because $\mathbb{P}_M$ is $ccc$. This implies that $\dot C \subseteq M$ and since $M$ is closed under $\omega$-sequences, $\dot C \in M$. Thus, $\val_{G_M}(\dot C) \in M[G_M]$.

It remains to prove that $\val_{G_M}(\dot u_M)$ is {\it ultra} over $M[G_M]$. Note that $\dot u$ is forced to be an ultrafilter on $\omega$, that is, $1 \Vdash \forall C \in [\omega]^\omega \ (C\in \dot u \text{ or } \omega \setminus C \in \dot u)$. 
As $\dot u \in M$ and $\Vdash$ is definable within $M$, the formula ``for every $\mathbb{P}_M$-name $\dot C$ for a subset of $\omega$, $1 \Vdash_{\mathbb{P}_M} \dot C \in u \vee \dot{\omega \setminus C} \in \dot u$'' holds in $M$. 
Thus, $M[G_M]$ satisfies that $\val_{G_M}(\dot u_M)$ is an ultrafilter on $\omega$, and so does $V[G_M]$ since $\mathcal{P}(\omega) \cap V[G_M] \subseteq M[G_M]$. This finishes Fact 1.
\medskip

In the following we will see that the sequence $\langle \val_{G_M} ( x_{\dot a_\alpha}) : \alpha\in  \omega_1\rangle$ converges to a unique point in the radial closure of $\omega$. Work in $V[G_M]$.

\begin{fact}
If $\Lambda$ is any uncountable subset of $\omega_1$, then there is a $\delta <\omega_1$ such that $\bigcup \{ \val_{G_M}(\dot a_\alpha) :\alpha \in \Lambda\cap \delta\}$ is an element of $\val_{G_M}(\dot u_M)$. 
\end{fact}

The set $\bigcup \{ \val_{G_M}(\dot a_\alpha) : \alpha \in \Lambda \}$ is countable, of course there is a $\delta$ so that $U= \bigcup \{ \val_{G_M}(\dot a_\alpha) :\alpha \in \Lambda\cap \delta\} = \bigcup \{ \val_{G_M}(\dot a_\alpha) :\alpha \in \Lambda\}$,
and then there is an $\alpha\in \Lambda,$ large enough so that $U$ and $\omega\setminus U$ are the evaluations of some $\mathbb{P}_M$-names in $M_\alpha$. 
If $U$ is not in $\val_{G_M}(\dot u_M)$ then $\omega \setminus U \in \val_{G_M}(\dot u_M)$. 
Since $\dot a_\alpha \in M\cap \mathcal{A}(\dot u, \dot u \cap M_\alpha)$, this implies that $\val_{G_M}(\dot a_\alpha)$ is mod finite contained in $\omega \setminus U$, contradicting $\val_{G_M}(\dot a_\alpha) \subseteq^* U$. This concludes Fact 2.

\medskip
It follows that for each clopen set captured by $M$ ($\dot b \in \dot B_X \cap M$) there is $\beta <\omega_1$ so that for every $\alpha \in \Lambda \setminus \beta$, $\val_{G_M}(\dot x_{\dot a_\alpha}) \in \dot b$. That is, $\langle \val_{G_M}(\dot x_{\dot a_\alpha}) : \alpha\in\Lambda\rangle$ converges with respect to all $\dot b\in \dot B_X\cap M$, (i.e. we may simply consider those $\dot b\in \dot B_X\cap \dot u$. Fact 2 shows that all but countably many $\val_{G_M}(\dot a_\alpha)$ are mod finite contained in $\val_{G_M}(\dot b$)). 
\medskip

Now we must prove that this convergence property is preserved by the tail forcing $\mathbb{P}/\mathbb{P}_M$.
Let $\dot b$ be any member of $\dot B_X$. Note
that $\dot b$ is forced by 1 not to \textit{split} any $\dot a_\alpha$ (these are converging sequences).
Towards a contradiction, let us assume that there is some condition $p\in \mathbb{P}/\mathbb{P}_M$ that
forces ``$\dot b$ mod finite contains uncountably many $\dot a_\alpha$, and is mod finite disjoint from uncountably many $\dot a_\beta$''. 
For each $\gamma<\omega_1$, choose any extension $p_\gamma \in \mathbb{P}/\mathbb{P}_M$ of $p$ together with
$\gamma\leq\alpha_\gamma,\beta_\gamma $ so that
there is an $m_\gamma$ satisfying
$p_\gamma\Vdash \dot a_{\alpha_\gamma}\setminus \dot b\subseteq \check m_\gamma \ \
\mbox{and} \ \
  \dot a_{\beta_\gamma}\cap \dot b\subset \check m_\gamma~~.$
Choose an uncountable $\Lambda\subseteq\omega_1$ so that for all $\gamma,\eta\in \Lambda$,  $m := m_\gamma=m_\eta$ and $p_\gamma\not\perp p_\eta$ (here we have used the fact that $\mathbb{P}/\mathbb{P}_M$ is forced by 1 to have property K). 
Choose $\delta<\omega_1$ as in Fact 2 for the sequence $\{ \val_G(\dot a_{\alpha_\gamma}) : \gamma\in \Lambda\}$, and let $U=\bigcup\{\val_G(\dot a_{\alpha_\gamma}) : \gamma\in \Lambda\cap \delta\}$ (note that since $\dot a_\alpha$ are countable sets, $\val_{G_M}(\dot a_\alpha) = \val_G (\dot a_\alpha)$).
We know that $U$ is in $\val_{G_M}(\dot u_M)$ so we can choose $\gamma\in \Lambda \setminus \delta$ large enough and $k > m$ with $k\in \val_{G_M}(\dot a_{\beta_\gamma})\cap U$. 
Choose $\eta\in \Lambda\cap \delta$ such that $k\in \val_G(\dot a_{\alpha_\eta})$. Then on one hand, we have that $p_\eta \Vdash \check k \in \dot b$, and on the other hand $p_\gamma \Vdash \check k \notin \dot b$. That is, $p_\eta \perp p_\gamma$, this is the desired contradiction. It follows then that there is a $\mathbb{P}$-name $\dot x (\dot u, M)$ so that $V[G]\models \langle \val_{G}(\dot x_{\dot a_\alpha}) : \alpha < \omega_1 \rangle$   converges to $ \val_{G}( \dot x (\dot u, M))$.
\medskip

 As for the uniqueness, simply check that
if $\mathcal S_1 = \{\dot a_\alpha: \alpha\in\omega_1\}$ and $\mathcal S_2 = \{\dot c_\alpha : \alpha\in\omega_1\}$ are two such sequences, there is a third $\mathcal S_3=\{\dot d_\alpha : \alpha \in\omega_1\}$ (for example, alternate the sequences $\dot a_\alpha$ and $\dot c_\alpha$) satisfying that $\mathcal S_1\cap \mathcal S_3$ and $\mathcal S_2\cap \mathcal S_3$ are both uncountable and have the same limits.
\end{proof}

\begin{remark}\label{Notation:X}
For larger models $M$ (that is, for $rk_{T_\kappa} (M) > 1$) we want to define analogues of $\dot x(\dot u,M)$. 
This definition will depend on the cofinality of $|M|$. 
When the cofinality is $\omega$, we will rather define an entire family $\mathcal X(\dot u, M)$ consisting of limits of converging $\omega$-sequences of the form $\langle \dot x(\dot u, M_n) : n\in \dot L\rangle$ where the sequence $\langle M_n : n\in\omega\rangle$ is an increasing chain that unions up to $M$ and $\dot L$ is any $\mathbb{P}$-name of an infinite subset of $\omega$. 
In the case when $\lambda = |M|$ is the successor of a cardinal with cofinality $\omega$, then $\mathcal X(\dot u, M)$ will be $\{ \dot x(\dot u, M)\}$ but its definition will be as a $\lambda$-limit of a  choice of members of $M\cap \mathcal X(\dot u,M_{\alpha\!+1})$ where $\{ M_{\alpha\!+1} : \alpha < \mbox{cf}(\lambda)\}$ is an increasing chain for $M$.
Finally, when $\lambda=|M|$ is any other cardinal, then $\mathcal X(\dot u, M) = \{ \dot x(\dot u,M)\}$ should be the limit of the sequence $\langle \dot x(\dot u,M_{\alpha\!+1}) : \alpha < \mbox{cf}(\lambda) \rangle$ for an increasing chain $\{M_\alpha : \alpha < \mbox{cf}(\lambda)\}$ for $M$.
Proving that these sequences radially converge within
$V[G_M]$ is not difficult, but we must again prove that
$\mathbb{P}/\mathbb{P}_M$ will preserve this convergence. 
\end{remark}

\begin{definition}\label{d:*}
Suppose $\kappa$ is a cardinal and that we have constructed a tree $T_\kappa$ as in Subsection \ref{SubTrees}. For $s\in T_\kappa$ we define the following statement:
\smallskip

($\star$)$_s$ if $\mbox{cf}(\kappa_s)>\omega$, then $1 \Vdash \dot x(\dot u, M_s)$ is in the closure of the limit points of members of $\mathcal A(\dot u,\mathcal W)\cap M_s$, where $\mathcal{W} \subseteq \dot u \cap M_s$ is countable.
\end{definition}

\begin{lemma}\label{l:*}
Fix $t\in T_\kappa$ and suppose \textnormal{($\star$)}$_s$ holds for every $s \supsetneq t$.
Assume that $\kappa_t>\aleph_1$ has uncountable
cofinality  and that $\dot b\in \dot B$.
Then the set of  $\gamma\in \kappa_t$ for which  there are a $p_\gamma$ and  pairs   $\alpha_\gamma,\dot x_{\gamma,0}$ and  $\beta_\gamma,\dot x_{\gamma,1}$ satisfying
  \begin{enumerate}
  \item $\gamma \leq \alpha_\gamma\leq \beta_\gamma<\kappa_t$,
  \item $  \dot x_{\gamma,0}$ is in $\mathcal X(\dot \mu,M_{t^\frown
    (\alpha_\gamma +1)})$,
  \item $  \dot x_{\gamma,1}$ is in $\mathcal X(\dot \mu,M_{t^\frown
    (\beta_\gamma +1)})$,
  \item $p_\gamma\Vdash \dot b\in \dot x_{\gamma,0}$,
    \item $p_\gamma\Vdash \dot b\notin \dot x_{\gamma,1}$,    
  \end{enumerate}
is bounded in $\kappa_t$.
\end{lemma}

\begin{proof} 
Assume that the sequence
$\mathcal S =\langle \{ p_\gamma,\alpha_\gamma,\beta_\gamma, \dot x_{\gamma,0}, \dot x_{\gamma,1}\} : \gamma \in \Gamma\rangle$ is a collection satisfying items (1)-(5) of the statement of the Lemma. We can further assume that for consecutive $\gamma < \gamma'$ in $\Gamma$, $\gamma < \alpha_\gamma < \beta_\gamma < \gamma '$.
Towards a contradiction, assume that $\Gamma$ is cofinal in $\mbox{cf}(\kappa_t)$. 
Fix any model $\bar M\prec H(\theta)$ of cardinality $\aleph_1$, closed under $\omega$-sequences, and satisfying that $\{ \dot u, \mathbb{P}, \dot b, \mathcal S, \Gamma, T_\kappa, \{ M_{t^\frown \xi \!+1 }: \xi <\mbox{cf}(\kappa_t)\} \} \subseteq \bar M$. 
     
Let $\lambda = \sup (\bar M\cap \kappa_t)$; $\bar M \cap \Gamma$ is cofinal in $\lambda$. Let $\{ \dot U_\delta : \delta < \omega_1\}$   be an   enumeration for $\dot u\cap (\bar M\cap M_t)$. 
By induction on $\delta\in\omega_1$, choose a strictly increasing sequence $\{ \gamma_\delta : \delta < \omega_1 \} \subseteq \Gamma\cap \bar M$ so that $\mathcal W_\delta =     \{ \dot U_\beta : \beta <\delta\}$     is an element   of $M_{t^\frown (\gamma_\delta+1)}$. Since $\bar M^\omega \subseteq \bar M$, $\mathcal{W}_\delta \in \bar M\cap M_{t^\frown (\gamma_\delta+1)}$. 
Observe that $\gamma_\delta, \alpha_{\gamma_\delta}, \beta_{\gamma_\delta} \in \bar M$ and therefore $M_{t^\frown (\gamma_\delta+1)}, M_{t^\frown (\alpha_{\gamma_\delta}+1)}$ and $M_{t^\frown (\beta_{\gamma_\delta}+1)}$ are also in $\bar M$, for $\delta < \omega_1$.
     
Fix any $\delta \in \omega_1$. We want to pick from $\mathcal A( \dot u,\mathcal  W_\delta)\cap (\bar M\cap M_t)$ sequences that are almost contained in $\dot b$ and sequences that are almost disjoint from $\dot b$, this will lead to a contradiction. To do so, we have two cases for $\mbox{cf}(\kappa_{t^\frown (\alpha_{\gamma_\delta}+1)})$.

\begin{case1}\label{Case1}
$\mbox{cf} (\kappa_{t^\frown (\alpha_{\gamma_\delta}+1)}) > \omega$. 
\end{case1}

By the definition of $\mathcal{X}(\dot u, M_{t^\frown (\alpha_{\gamma_\delta}+1)})$, in this case we know that this set consists of a unique point, thus 1 forces that $\dot x_{\gamma_\delta, 0}$ coincides with $\dot x(\dot u, M_{t^\frown (\alpha_{\gamma_\delta}+1)})$.
Using (2), (4) and ($\star$)$_{t^\frown (\alpha_{\gamma_\delta}+1)}$ we have that $H(\theta)$ satisfies that $p_{\gamma_\delta} \Vdash \dot x(\dot u, M_{t^\frown (\alpha_{\gamma_\delta}+1)})$ is in the closure of the limit points of members of $\mathcal A(\dot u,\mathcal W_{\delta})\cap M_{t^\frown (\alpha_{\gamma_\delta}+1)}$, and $p_{\gamma_\delta} \Vdash \dot b \in \dot x(\dot u, M_{t^\frown (\alpha_{\gamma_\delta}+1)})$. 
Hence, $H(\theta)$ satisfies that there is a convergent sequence and an extension of $p_{\gamma_\delta}$ that forces the sequence to be almost contained in $\dot b$. Since all required parameters for reflection ($p_{\gamma_\delta}, \dot u, \mathcal{W}_\delta, M_{t^\frown (\alpha_{\gamma_\delta}+1)}$) are in $\bar M$,
$\bar M$ also satisfies there is $\dot a_{\alpha_{\gamma_\delta}}\in \mathcal A( \dot u,\mathcal W_\delta)\cap M_{t^\frown (\alpha_{\gamma_\delta}+1)}$ and a $q'_{\gamma_\delta}<p_{\gamma_\delta}$ so that $q'_{\gamma_\delta} \Vdash \dot a_{\alpha_{\gamma_\delta}} \subseteq^* \dot b$.

\medskip
By the property \ref{SubTrees}.(6) of $T_\kappa$ we have $\kappa_{t^\frown (\alpha_{\gamma_\delta}+1)} = \kappa_{t^\frown (\beta_{\gamma_\delta}+1)}$ and this implies $\mbox{cf} (\kappa_{t^\frown (\beta_{\gamma_\delta}+1)}) > \omega$. So, similarly using (3), (5) and ($\star$)$_{t^\frown (\beta_{\gamma_\delta}+1)}$, $\bar M$ satisfies that there is $\dot a_{\beta_{\gamma_\delta}}\in \mathcal A( \dot u,\mathcal W_\delta)\cap M_{t^\frown (\beta_{\gamma_\delta}+1)}$ and $q_{\gamma_\delta} <q'_{\gamma_\delta}$ such that $q_{\gamma_\delta} \Vdash \dot a_{\alpha_{\gamma_\delta}} \subseteq^* \dot b \text{ and } \dot a_{\beta_{\gamma_\delta}}\cap \dot b =^*\emptyset$.

\begin{case2}
$\mbox{cf} (\kappa_{t^\frown (\alpha_{\gamma_\delta}+1)}) = \omega$. 
\end{case2}
Denote by $\langle M_n : n\in \omega \rangle$ the sequence $\langle M_{t^\frown (\alpha_{\gamma_\delta}+1 ^\frown n)} : n\in \omega \rangle$ for $M_{t^\frown (\alpha_{\gamma_\delta}+1)}$ (recall each $M_n$ has regular cardinality, \ref{SubTrees}.(5)). 
Since $M_{t^\frown (\alpha_{\gamma_\delta}+1)}$ is the $\subseteq$-increasing union of the $M_n$'s (\ref{SubTrees}.(4)) and $\mathcal{W}_\delta \in M_{t^\frown (\gamma_\delta +1)} \subseteq  M_{t^\frown (\alpha_{\gamma_\delta}+1)}$, there is $k\in \omega$ such that for all $n\geq k$, $\mathcal{W}_\delta \in M_n$ and hence $\mathcal{W}_\delta \subseteq M_n$.
Also, as $\dot x_{\gamma_\delta, 0}$ is an element of $\mathcal{X}(\dot u, M_{t^\frown (\alpha_{\gamma_\delta}+1)}) \cap \bar M$ there is a $\mathbb{P}$-name, $\dot L \in \bar M$, of an infinite subset of $\omega$ such that $p_{\gamma_\delta}$ forces that the sequence $\langle \dot x(\dot u, M_n) : n\in \dot L\rangle $ (which is an element of $\bar M$) converges to $\dot x_{\gamma_\delta, 0}$.
So we can choose large enough $n\in \omega$ and $q'_{\gamma_\delta} < p_{\gamma_\delta}$ such that $\mathcal{W}_\delta \in M_n$ and $q'_{\gamma_\delta} \Vdash \dot b \in \dot x(\dot u, M_n)$. 
Again, all required parameters are in $\bar M$ and since $|M_n|$ has uncountable cofinality, repeating the arguments as in Case One we can obtain, within $\bar M$, elements $\dot a_{\alpha_{\gamma_\delta}}, \dot a_{\beta_{\gamma_\delta}}$ and $q_{\gamma_\delta}$ as above. Case Two is finished.

\bigskip
We have obtained the collections 
$\langle q_{\gamma_\delta} \in \bar M \cap \mathbb{P} : \delta < \omega_1 \rangle$, 
$\langle a_{\alpha_{\gamma_\delta}} \in \mathcal{A} (\dot u, \mathcal{W}_\delta) \cap (\bar M \cap M_{t^\frown (\alpha_{\gamma_\delta}+1)}) : \delta < \omega_1 \rangle$
and $\langle a_{\beta_{\gamma_\delta}} \in \mathcal{A} (\dot u, \mathcal{W}_\delta) \cap (\bar M \cap M_{t^\frown (\beta_{\gamma_\delta}+1)}) : \delta < \omega_1 \rangle$
such that for every $\delta < \omega_1$, $q_{\gamma_\delta} < p_{\gamma_\delta}$ and $q_{\gamma_\delta} \Vdash a_{\alpha_{\gamma_\delta}} \subseteq^* \dot b \text{ and }
a_{\beta_{\gamma_\delta}}\cap \dot b =^*\emptyset$.
Note that $\bar M \cap M_t$ has cardinality $\aleph_1$ and is closed under $\omega$-sequences (this follows by our assumption on $\bar M$ and \ref{SubTrees}.(6)). 
Fix any $\subseteq$-chain $\langle \bar M_\xi :\xi < \omega_1 \rangle$ of countable elementary submodels that unions up to the model $\bar M \cap M_t$ such that for every $\xi\in \omega_1$, $\{q_{\gamma_\xi}, \dot a_{\alpha_{\gamma_\xi}}, \dot a_{\beta_{\gamma_\xi}}, \mathcal{W}_\xi \} \subseteq \bar M_{\xi +1}$.
Now observe that $\bigcup_{\delta \in \omega_1} \mathcal{W}_\delta = \dot u \cap (\bar M \cap M_t) = \bigcup_{\delta\in \omega_1} \dot u \cap \bar M_\delta$. 
Hence there is a c.u.b. $C\subseteq \omega_1$ such that for every $\delta \in C$, $\dot u \cap \bar M_\delta = \mathcal{W}_\delta$. 
Consider a $\mathbb{P}$-name $\dot S$ for the set $\{ \delta \in C : q_{\gamma_\delta} \in G \}$ (recall that $G$ is a $\mathbb{P}$-generic filter). 
Using the fact that $\mathbb{P}$ is finally property $K$ (in fact, only by $ccc$), $\dot S$
is forced by some condition in $G$ to be uncountable.
So, in $V[G]$ we have $\omega_1$-many sequences $\val_G ( \dot a_{\alpha_{\gamma_\delta}} )$ that are almost contained in $\val_G (\dot b)$ and $\omega_1$-many sequences $\val_G (a_{\beta_{\gamma_\delta}})$ that are almost disjoint from $\val_G (\dot b)$ which contradicts Lemma \ref{KeyLemma} for the model $\bar M \cap M_t$.
\end{proof}

\medskip
Let us note that Lemma \ref{KeyLemma} and Lemma \ref{l:*} imply the following: if $\mbox{cf}(\kappa_t) > \omega$, then for any choice of an element $\dot x_\gamma$ in $\mathcal{X} (\dot u, M_{t^\frown (\gamma +1)})$, $\gamma < \mbox{cf}(\kappa_t)$, we have that the sequence $\langle \dot x_\gamma : \gamma < \mbox{cf}(\kappa_t) \rangle$ converges to a unique point (that is, $\mathcal{X} (\dot u, M_t) = \{ \dot x(\dot u, M_t) \}$).

We can think of Lemma \ref{l:*} as a generalization of Lemma \ref{KeyLemma} and we use it in the next to lift ($\star$) up to higher levels.

\begin{lemma}\label{l:**}
If \textnormal{($\star$)}$_s$ holds for every $s\in T$ with $t\subsetneq s$, then \textnormal{($\star$)}$_t$ holds.
\end{lemma}
\begin{proof}
The case when $\kappa_t$ has countable cofinality is straightforward by sequential compactness. Thus let us assume $\kappa_t$ has uncountable cofinality. Now fix a countable family $\mathcal{W} \subseteq \dot u \cap M_t$. 
We want to prove that $1$ forces ``every neighborhood around $\dot x(\dot u, M_t)$ contains an element of $\mathcal A(\dot u,\mathcal W)\cap M_t$''. 
So, pick any $\dot b \in \dot B$ such that $1\Vdash \dot b \in \dot x(\dot u, M_t)$. 

The increasing family $\langle M_{t^\frown (\gamma +1)} : \gamma < \mbox{cf}(\kappa_t) \rangle$ unions up to $M_t$, so by the argument preceding this lemma for any choice for $\dot x_\gamma$ in $\mathcal{X}(\dot u, M_{t^\frown (\gamma +1)})$, $\gamma < \mbox{cf}(\kappa_t)$, the sequence $\langle \dot x_\gamma : \gamma < \mbox{cf}(\kappa_t) \rangle$ converges to $\dot x(\dot u, M_t)$.
By $ccc$ and because $\mbox{cf}(\kappa_{t}) > \omega$, $1 \Vdash$ ``$\dot b \in \dot x_\gamma$ for all but fewer than $\mbox{cf} (\kappa_t)$-many $\gamma$'s''.
Take any large enough $\gamma$ so that $\mathcal{W} \in M_{t^\frown (\gamma+1)}$ and $1\Vdash \dot b \in \dot x (\dot u, M_{t^\frown (\gamma+1)})$.

\begin{case1}
$\mbox{cf}(\kappa_{t^\frown (\gamma +1)}) > \omega$.
\end{case1}
Lemma \ref{l:*} implies that $\dot x_\gamma = \dot x (\dot u, M_{t^\frown (\gamma +1)})$. 
Next, ($\star$)$_{t^\frown (\gamma +1)}$ implies that there is $\dot a_\gamma \in \mathcal A(\dot u,\mathcal W)\cap M_{t^\frown (\gamma +1)} \subseteq \mathcal A(\dot u,\mathcal W)\cap M_t$ such that $1 \Vdash \dot a_\gamma \subseteq^* \dot b$, as desired.

\begin{case2}
$\mbox{cf}(\kappa_{t^\frown (\gamma +1)}) = \omega$.
\end{case2}
Fix any condition $q\in \mathbb{P}$. 
Let $\dot L$ be a $\mathbb{P}$-name for a subset of $\omega$ such that $1 \Vdash$ ``$\langle \dot x(\dot u, M_{t^\frown (\gamma +1 ^\frown n)}) : n\in \dot L \rangle$ converges to $\dot x_\gamma$''. 
Choose a large enough $n\in \omega$ and a condition $p < q$ such that $\mathcal{W} \in M_{t^\frown (\gamma+1 ^\frown n)}$ (hence $\mathcal{W} \subseteq M_{t^\frown (\gamma +1^\frown n)}$) and $p\Vdash \dot b \in \dot x(\dot u, M_{t^\frown (\gamma +1^\frown n)})$. 
Because $\kappa_{t^\frown (\gamma+1 ^\frown n)}$ has uncountable cofinality (\ref{SubTrees}.(6)) we can apply ($\star$)$_{t^\frown (\gamma+1 ^\frown n)}$ and repeat the arguments of Case One to get $\dot a_\gamma \in \mathcal A(\dot u,\mathcal W)\cap M_{t^\frown (\gamma +1^\frown n)}$ such that $p \Vdash \dot a_\gamma \subseteq^* \dot b$. Since this applies for every $q$, we have proved that there is $\dot a_\gamma \in \mathcal A(\dot u,\mathcal W)\cap M_{t^\frown (\gamma+1 ^\frown n)} \subseteq A(\dot u,\mathcal W)\cap M_{t}$ such that $1 \Vdash \dot a_\gamma \subseteq^* \dot b$. This concludes Case Two as well as the proof of the lemma.
\end{proof}

Now we prove our main result.

\begin{proof}[Proof of Main Theorem \ref{MainTheorem}]\
Fix a large enough cardinal $\kappa$ such that $\max\{|\mathbb{P}|, |\dot u|\} \leq \kappa$ and $\kappa^\omega = \kappa$. All elementary submodels are substructures of $H(\theta)$ where $\theta = 2^\kappa$.
From $\square$ and GCH we can get a tree $T_\kappa$ as in Subsection \ref{SubTrees} so that for every maximal node $t\in T_\kappa$, $\{ \dot u, \dot B, \mathbb{P}\} \subseteq M_t$. 
\medskip

Now we begin with the induction over $rk_{T_\kappa}$.
For the base case $rk_{T_\kappa}(t)=1$, we have that $M_t$ has cardinality $\omega_1$ and by Lemma \ref{KeyLemma} we have our definition of $\dot x(\dot u, M_t)$ which is in the radial closure of $\omega$ and satisfies $(\star)_t$.
Now, Lemma \ref{l:**} implies that the induction holds all the way up to $t= \emptyset$. 
We claim that $\dot x (\dot u, M_\emptyset )$ is the $\dot u$-limit and is in the radial closure of $\omega$. 
In fact, this follows from ($\star$)$_\emptyset$ and the fact that $\dot u \subseteq  M_\emptyset$. 
That is, if $\dot b \in \dot u \cap \dot B_X$ is any neighborhood of  $\dot x (\dot u, M_\emptyset )$ (a member of $\dot x (\dot u, M_\emptyset )$) then 1 forces that $\dot b$ almost contains an element of $\mathcal A(\dot u, \{ \dot b \}) \cap  M_\emptyset$. 
To see that $\dot x(\dot u, M_\emptyset) \in \omega^{(r)}$ note that it is the limit point of the sequence $\langle \dot x_{\alpha} : \alpha < \mbox{cf} (| M_\emptyset|)\rangle$ for any choice of $\dot x_{\alpha}$ in $\mathcal{X}(\dot u, M_{\alpha+1})$ (see Remark~\ref{Notation:X}).
\end{proof}


\section{The Pseudoradial Number $\mathfrak{pse}$ and Products}\label{SectionProducts}
In this section we analyze how weak pseudoradiality {\it interacts} with the cardinal $\mathfrak{pse}$. We prove Theorems \ref{t:pseudoradialiff} and \ref{t:product} towards the end of this section. 
We may assume spaces are 0-dimensional because of Theorem \ref{t:Sapirovskii}, so we work on $2^\kappa$ instead of $[0,1]^\kappa$.
%

\medskip

In the following we slightly modify an important result by Bella, Dow and Tironi.
We include the proof for the sake of completeness.

\begin{lemma}\label{LemmaPSE}
Suppose that a compact space $X$ cannot be mapped onto $2^\mathfrak{pse}$ and that $\mathfrak{pse}$ is regular. Then there is $\lambda < \mathfrak{pse}$ and a sequence $\{ H_n : n\in \omega \}$ of non-empty closed $G_\lambda$-sets in $X$ that forms a $\pi$-net for some point $x\in X$.
\end{lemma}
\begin{proof}

Suppose that the statement fails. 
We follow the induction as in \cite{bella2001pseudoradial}. 
Start with a countable family $\{ H(n, 0): n\in \omega \}$ of pairwise disjoint closed subsets of $X$.
Inductively we will choose an independent family $\{ B_\mu : \mu < \mathfrak{pse} \}$ of clopen sets of $X$ (i.e. a family of clopen sets $\mathcal{B}$ such that for any finite subcollection $A_0, \ldots, A_n, B_0, \ldots, B_m \in \mathcal{B}$, the set $(\bigcap_{i\leq n} A_i) \cap (\bigcap_{i\leq m} X\setminus B_i)$ is clopen) and closed sets $\{ H(n, \mu) : \mu < \mathfrak{pse} \}$, $n \in \omega$, such that for each $\mu < \mathfrak{pse}$, $H(n, \mu +1)$ is set equal to either $H(n,\mu)\cap B_\mu$ or $H(n,\mu) \setminus B_\mu$. If $\mu$ is limit, set $H(n,\mu) = \bigcap_{\beta < \mu} H(n, \beta)$. 
Also choose $\sigma_\mu \in Fn(\mu, 2)$, a finite partial function from $\mu$ to $2$, such that the following formula ($\odot_\mu$) holds

\smallskip
($\odot_\mu$) \ \ for all $\tau \in Fn(\mu, 2)$ such that $\sigma_\beta = \sigma_\mu$ for each $\beta \in dom(\tau)$

\hspace{1cm} $(|\{ n\in \omega : H(n, \mu \!+1) \subseteq B_{\sigma_\mu} \cap (B_\tau \cap B_\mu) \}| = \aleph_0$ and 

\hspace{1.2cm} $|\{ n\in \omega : H(n, \mu \!+1) \subseteq B_{\sigma_\mu} \cap (B_\tau \setminus B_\mu) \}| = \aleph_0$).
\smallskip

Here, if $\tau \in Fn(\mu,2)$ set $B_\tau = \bigcap_{\alpha \in dom(\tau)} B_\alpha^{\tau(\alpha)}$, where $B_\alpha^{\tau(\alpha)} = B_\alpha$ if $\tau(\alpha) = 1$ and $B_\alpha^{\tau(\alpha)} = X \setminus B_\alpha$ if $\tau(\alpha) = 0$.
\medskip

Suppose we have constructed the sets $\{ H(n,\mu) : n\in \omega \}$ and $\{ B_\alpha : \alpha < \mu \}$. 
We have to find $B_\mu$ and $H(n,\mu+1)$ for each $n\in \omega$. 
By the assumption above for each $\alpha < \mu$, $H(n,\mu)$ is either contained in, or disjoint from $B_\alpha$. 
For $\alpha < \mu$ let $Y_\alpha = \{ n\in \omega : H(n,\mu) \subseteq B_\alpha \}$. 
By ($\odot_\mu$) each $Y_\alpha$ is infinite. 
Let $\mathcal{Y}_\mu$ be the Boolean subalgebra of $\mathcal{P}(\omega)$ generated by $\{ Y_\alpha : \alpha < \mu \}$. 
The Stone space of $\mathcal{Y}_\mu / \text{fin}$, $S(\mathcal{Y}_\mu / \text{fin})$, is a compactification of $\omega$, hence it is the image of the remainder $\omega^* = \beta \omega \setminus \omega$ under the natural map, namely $f$. 
Apply Zorn's Lemma to $\mathcal{C} = \{ K : K \subseteq \omega^* \text{ is closed and } f\restriction K \text{ is onto} \}$ to find a closed $K_\mu$ that is $\supseteq$-minimal. That is, $f_\mu = f\restriction K_\mu$ is an irreducible map from $K_\mu$ onto $S(\mathcal{Y}_\mu / \text{fin})$. 
\smallskip

In the following we find $B_\mu$. Let $F_\mu$  be the filter of those $A\subseteq \omega$ such that $K_\mu$ is contained in $A^*$. Define  $H_\mu$ to be the intersection of the family $\{   cl( \bigcup \{ H(n,\mu)  : n \in  A \}  )   :    A\in F_\mu  \}$ (in essence, $H_\mu$ is the non-empty set of the $K_\mu$-limits of the $H(n,\mu)$'s). 

We claim that there is a clopen $B$ in $X$ such that $K_\mu \cap (Z_B)^* \neq \emptyset$, where $Z_B = \{ n :   B \text{ splits }   H(n, \mu)  \}$ ($S$ splits $A$ means that both $A\cap S$ and $A \setminus S$ are non-empty). 
Once proved our claim we will let $B_\mu = B$. 
Assume towards a contradiction that for each clopen $B$, $(Z_B)^*$  misses $K_\mu$ which is the same as saying that $Z_B$ is in the dual ideal of $F_\mu$. 
By the assumption that the collection $\{ H(n,\mu) : n\in \omega \}$ is not a $\pi$-net for any point $x\in X$, for each $x$ in $H_\mu$ choose a clopen neighborhood $B_x$ of $x$ that contains no $H(n,\mu)$. 
By compactness, let $B_{x_1},  B_{x_2} , \ldots  ,  B_{x_m}$  be a finite cover of $H_\mu$ consisting of such $B_x$'s. 
Then there is an $A$ in $F_\mu$ such that none of $B_{x_1} , \ldots,  B_{x_m}$ splits $H(n, \mu)$  for any $n\in A$ (otherwise if there is $i \leq m$ such that $B_{x_i}$ splits $H(n,\mu)$ for all $n\in A$, then $A=Z_{B_{x_i}} \in F_\mu$, which is not possible).
However, one of the $B_{x_i}$'s   must hit  at least one of the $H(n,\mu)$'s  for $n\in A$ but this means that  one of those $H(n,\mu)$ is contained in one of those $B_{x_i}$. This is the desired contradiction.
\smallskip

Now that we have found $B_\mu$, we find $\sigma_\mu$. Observe that $K_\mu \cap \overline{Z_\mu}$ is clopen relative to $K_\mu$ hence $f_\mu [ K_\mu \cap Z_\mu]$ has interior in $S(\mathcal{Y}_\mu / \text{fin})$. This implies that there is $\sigma_\mu \in Fn(\mu,2)$ such that the closure of the set $Y_{\sigma_\mu} := \bigcap \{ Y_\alpha : \sigma_\mu (\alpha) = 1  \} \cap \bigcap \{ \omega \setminus Y_\alpha : \sigma_\mu (\alpha ) = 0 \}$ in $S(\mathcal{Y}_\mu / \text{fin})$ is contained in $f_\mu [ K_\mu \cap Z_\mu]$. This implies that $K_\mu$ is disjoint from the closure of $Y_{\sigma_\mu} \setminus Z_\mu$ in $\omega^*$,
and as a consequence of this fact, for all $Y\in \mathcal{Y_\mu}$, if $Y\cap Y_{\sigma_\mu}$ is infinite, $Y\cap (Y_{\sigma_\mu} \cap Z_\mu)$ is also infinite. 
\smallskip

Let's now find $H(n, \mu +1)$, for each $n\in \omega$. Set $J_\mu = \{ \beta : \sigma_\beta = \sigma_\mu \}$. By inductive assumption, $\{ Y_\beta \cap Y_{\sigma_\mu} : \beta \in J_\mu \}$ is an independent family on $Y_{\sigma_{\mu}}$. To see this, take any $\tau \in Fn(J_\mu, 2)$ and let $\mu' = \mbox{max} \mathop{dom} (\tau)$, $\mu' < \mu$. Then the formula $|Y_\tau \cap Y_{\sigma_\mu}| = \aleph_0$ follows from the relevant clause $(\odot_{\mu'})$ (depending upon the value of $\tau(\mu')$). 
In addition, $\{ Y_\beta \cap (Y_{\sigma_\mu} \cap Z_\mu) : \beta \in J_\mu \}$ is a non-maximal independent family (because $\mu < \mathfrak{pse \leq s \leq i}$) on $Y_{\sigma_\mu} \cap Z_\mu$, so we can choose $Y \subseteq Y_{\sigma_\mu} \cap Z_\mu$ such that $\{ Y_\beta : \beta \in J_\mu \} \cup \{ Y\}$ is independent on $Y_{\sigma_\mu} \cap Z_\mu$.
Set 
$H(n,\mu+1)$ to be $H(n,\mu)\cap B_\mu$ if $n\in Y$, $H(n,\mu)\setminus B_\mu$ if $n\in Z_\mu \setminus Y$, or $H(n,\mu)$ if $n\notin Z_\mu$. 
Finally redefine $B_\mu$ to be $B_\mu \cap B_{\sigma_\mu}$. This completes the induction.
\smallskip

To finish, observe that, by the pressing down lemma, there would be $\mathfrak{pse}$-many $\mu$ with the same value for $\sigma_\mu$ and this would result on an $\mathfrak{pse}$-sized independent family of clopen subsets of $X$. Then $X$ would map onto $2^\mathfrak{pse}$, contradiction.
\end{proof}

\begin{definition}
We say that a subset $A$ is {\it $G_\lambda$-dense in its closure} if for every $G_\lambda$-set $H\subseteq \overline{A}$, $A\cap H \neq \emptyset$.
\end{definition}

Note that every $G_\lambda$-set contains a closed $G_\lambda$-set. Also it can be easily checked that if $A$ is a radially closed  subset of a sequentially compact space, then $A$ is $G_\delta$-dense in its closure. 

For a cardinal $\kappa$, $\kappa^- = \kappa$ if $\kappa$ is limit, otherwise $\kappa^-$ is the predecessor of $\kappa$.

\begin{lemma}\label{LemmaLambda}
Let $X$ be a compact weakly pseudoradial space which cannot be mapped onto $2^{\mathfrak{pse}}$. 
Suppose that $A\subseteq X$ is radially closed with $\lambda = \lambda(A,X) \geq \mathfrak{pse}^-$ and assume $\mathfrak{pse}$ is regular.
Then,  $A$ is $G_{\gamma}$-dense in $\overline{A}$ for each $\gamma \leq \lambda$.
\end{lemma}
\begin{proof}
Since a $G_\gamma$ set is also $G_\eta$ when $\gamma \leq \eta$, it suffices to prove the result for $\gamma = \lambda = \lambda (A, X)$.
Let $H$ be a closed $G_{\lambda}$-set in $\overline{A}$. We can get a sequence $\{ W_\alpha : \alpha < \lambda \}$ of closed sets such that $W_\alpha$ is the intersection of at most $|\alpha| \cdot \aleph_0$-many open sets and the sequence intersects down to $H$. Let us note that $H$ has no isolated points, otherwise there would be a sequence of elements in $A$ converging to such points, contradicting radial closedness. In particular,  $H$ is infinite. 

The set $H$ inherits from $X$ compactness and cannot be mapped onto $2^{\mathfrak{pse}}$. By Lemma~\ref{LemmaPSE} applied to $H$, there is a collection $\{ H_n : n\in \omega \}$ of closed $G_{\gamma}$-sets in $H$, for some $\gamma < \mathfrak{pse}$, that forms a $\pi$-net around a point $x\in H$. Since $\gamma \leq \mathfrak{pse}^- \leq \lambda$, each set $H_n$ is a closed $G_\lambda$-set in $H$. For each $n\in \omega$, we can choose a collection of closed sets $\{ V_\alpha (n): \alpha < \lambda \}$ in $\overline{A}$ whose intersection with $H$ is $H_n$ and $V_\alpha(n)$ is the intersection of at most $|\alpha|\cdot \aleph_0$ open sets. 
For $n\in \omega$ and $\alpha < \lambda$, define $W_\alpha^n = W_\alpha \cap V_\alpha (n)$.
By the minimality of $\lambda$, $W_\alpha^n \cap A \neq \emptyset$,  so we choose a point $x(\alpha,n)$ in $W_\alpha^n \cap A$, for each $\alpha <\lambda$, $n\in \omega$.

Pick an ultrafilter $u$ on $\omega$ such that $x$ is the $u$-limit of the sequence $\{ H_n : n\in \omega\}$.
Let  $x^u_\alpha$  denote the $u$-limit of the set $\{ x(\alpha,n) : n\in\omega \}$. 
As $X$ is weakly pseudoradial, the radial closure of $\{ x(\alpha,n) : n\in\omega \}$ is closed and we are assuming that $A$ is radially closed, so $x_\alpha^u$ is in $A$. 

It is easy to see now that $\{ x^u_\alpha : \alpha< \lambda \}$ converges to $x$, therefore $x\in A$. Thus, $H \cap A \neq \emptyset$ as claimed.
\end{proof}

For Theorem \ref{t:product} we need the following lemma. The authors apologize if the corresponding reference is missing; a proof is given.

\begin{lemma}\label{LemmaProductCannotBeMapped}
If $X$ and $Y$ are compact spaces that do not map onto $[0,1]^\kappa$,
then neither does the product $X\times Y$.
\end{lemma}
\begin{proof}
Towards a contradiction assume that $f$ is a continuous function from $X\times Y$ onto $[0,1]^\kappa$.
We can pass to a closed subset $F$ of $X\times Y$ so that $f[F] = \{0,1\}^\kappa$ and $f\restriction F$ is irreducible.  
Recall that every relatively open subset of $F$ contains the full preimage of some non-empty open subset of $2^\kappa$.
\smallskip

Denote by $\pi_X$ the canonical projection from $X\times Y$ to $X$.
Consider the closed subset $\pi_X[F]$ of $X$.
By Theorem \ref{t:Sapirovskii} we can choose
$x\in \pi_X[F]$ so that $\lambda_x = \pi\chi(x,\pi_X[F])<\kappa$.
Let $\{ U_\alpha : \alpha < \lambda_x\}$ be a family of open subsets of $X$ so that $\{ U_\alpha \cap \pi_X[F] : \alpha < \lambda_x\}$ is a relative local $\pi$-base at $x$.
For each $\alpha$, let $F[U_\alpha] = F\cap (U_\alpha\times Y)$. Choose a basic clopen $[\sigma_\alpha]\subseteq 2^\kappa$ so that $F_{\sigma_\alpha}=F\cap f^{-1}([\sigma_\alpha])$ is contained in $F[U_\alpha]$.
Choose any ultrafilter $\mathcal U$ on $\lambda_x$ that extends the 
neighborhood trace of $x$ on the family $\{ U_\alpha : \alpha\in \lambda_x\}$.
That is, for each open $x\in U\subseteq X$, the set $\{ \alpha <\lambda_x : U_\alpha \subseteq U\}$ is an element of $\mathcal U$. 
\smallskip

Now let $H_{\mathcal U}$ be the set of all $\mathcal U$-limits of the family $\{ F_{\sigma_\alpha} : \alpha < \lambda_x\}$. 
In other words, $z\in H_{\mathcal U}$ if and only if for each open $z\in U\times W\subseteq X\times Y$, the set $\{ \alpha :  F_{\sigma_\alpha}\cap  ( U\times W) \neq \emptyset\}$ is in the filter $\mathcal U$, or equivalently, for all $I \in \mathcal U$, $z$ is in the closure of  $\bigcup \{ F_{\sigma_\alpha}  : \alpha \in I\}$.
Note that $H_\mathcal{U} \subseteq F$ since $F$ is closed, and even more specifically, $H_{\mathcal U}$ is a subset of $F_x = F\cap (\{x\}\times Y)$.  

Let $J$ be the set of indices $\kappa\setminus \bigcup\{ \mathop{dom}(\sigma_\alpha) : \alpha <\lambda_x\}$ and let $\pi_J$ denote the projection of $2^\kappa$ onto $2^J$.  
Consider the set $(\pi_J\circ f)[  H_{\mathcal U} ]\subseteq 2^J$. 
This set is nowhere dense in $2^J$ since $\{x\}\times Y$ does not map onto $2^\kappa$. %
Then choose a non-empty clopen $[\tau]\subseteq 2^J$ such that $[\tau]\cap (\pi_J\circ  f)[H_{\mathcal U}]$ is empty.
Now consider $[\tau]$ as a subset of $2^\kappa$ (same as $\pi_J^{-1}([\tau])$).  
For each $\alpha<\lambda_x$,  $[\tau]\cap [\sigma_\alpha]$ is not empty. 
Also $f^{-1}([\tau \cup \sigma_\alpha])\cap F$ is a subset of $F_{\sigma_\alpha}$. The set of $\mathcal U$-limits, $H_{\tau,\mathcal U}$, of the family $\{ f^{-1}([\tau \cup \sigma_\alpha]) \cap F : \alpha < \lambda_x\}$ is a non-empty subset of $H_{\mathcal U}$. 
Clearly  $f[H_{\tau,\mathcal U}]\subseteq [\tau]$ and hence $(\pi_J\circ f)[H_{\tau,\mathcal U}]$ is non-empty. This contradicts that $(\pi_J\circ f)[H_{\mathcal U}]\cap [\tau]$ is empty.
\end{proof}

\begin{proof}[Proof of Theorem \ref{t:pseudoradialiff}]\
The forward implication is immediate. 
Now let us assume that $X$ cannot be mapped onto $2^\mathfrak{pse}$ and towards a contradiction suppose that $A$ is a radially closed non-closed subset of $X$. Consider a closed  $G_\lambda$ subset $H$ of $\overline{A} \setminus  A$ with $\lambda$  minimal. Let $\{ W_\alpha : \alpha < \lambda \}$  be the descending sequence of closed sets such that $W_\alpha$ is equal to the intersection of at most  $|\alpha|\cdot \aleph_0$ many open sets and $H$ equals the intersection.

If $\mathfrak{pse} = \aleph_1$, by \v Sapirovski\v \i's Theorem \ref{t:Sapirovskii} there is a point  $x$ in $H$ that has countable $\pi$-character. 
Let  $\{ H_n : n \in \omega\}$ be a local $\pi$-net for  $x$  where  each   $H_n$  is a closed $G_\delta$-set in  $H$. Choose any ultrafilter  $u$ on $\omega$ so that $x$ is the $u$-limit  of the sequence $\{ H_n  : n \in \omega \}$. 
Then we can simply choose   $G_\delta$-sets, $Z_n$,  so that $Z_n \cap H  =   H_n$. For each $\alpha  < \lambda$ and  $n\in\omega$,  choose a point $a(\alpha,n)$ in $Z_n \cap W_\alpha$.
Let  $x_\alpha$  denote the $u$-limit of the set $\{ a(\alpha,n) : n\in\omega \}$. Since $X$ is weakly pseudoradial and $A$ is radially closed then $x_\alpha$ is in $A$. It is easy to check that $\{ x_\alpha : \alpha < \lambda \}$ converges to  $x$, contradicting $A$ is non-closed.

If $\mathfrak{pse} = \aleph_1$, as $X$ is sequentially compact, the cofinality of $\lambda$ is uncountable. In particular, $\lambda \geq \omega_1 = \mathfrak{pse}^-$, therefore Lemma~\ref{LemmaLambda} applies. The set $A$ is $G_\lambda$-dense in $\overline{A}$ and must meet $H$, contradicting that $A$ is non-closed.
\end{proof}

\begin{proof}[Proof of Theorem \ref{t:product}]\
By Theorem \ref{t:pseudoradialiff} the spaces $X$ and $Y$ cannot be mapped onto $2^\mathfrak{pse}$. By Lemma \ref{LemmaProductCannotBeMapped}, $X\times Y$ cannot be mapped onto $2^\mathfrak{pse}$ either. Since we are assuming that $X\times Y$ is weakly pseudoradial, Theorem \ref{t:pseudoradialiff} applies again so the product is pseudoradial.
\end{proof}

\bibliographystyle{plain}
\bibliography{biblio}
\end{document}